\documentclass[11pt,hidelinks]{amsart}
\usepackage{amsmath,enumerate,amsfonts,amssymb,amsthm,verbatim,mathtools,textcomp,graphicx}
\usepackage[hypertexnames=false]{hyperref}

\newtheorem{theorem}{Theorem}
\newtheorem{lemma}[theorem]{Lemma}
\newtheorem{proposition}[theorem]{Proposition}
\newtheorem{corollary}[theorem]{Corollary}

\theoremstyle{definition}
\newtheorem{definition}[theorem]{Definition}
\newtheorem{remark}[theorem]{Remark}

\newcommand{\Section}[1]{\section{#1}\setcounter{theorem}{0}}

\newcommand{\KN}{\mathbin{\bigcirc\mspace{-15mu}\wedge\mspace{3mu}}}
\newcommand{\gabc}{{g_{abc}}}
\newcommand{\ii}{{\mathbf{i}}}
\newcommand{\jj}{{\mathbf{j}}}
\newcommand{\kk}{{\mathbf{k}}}
\newcommand{\N}{{\mathbb N}}
\newcommand{\Z}{{\mathbb Z}}
\newcommand{\HH}{{\mathbb H}}
\newcommand{\eps}{{\varepsilon}}
\renewcommand{\theta}{{\vartheta}}
\renewcommand{\phi}{{\varphi}}
\newcommand{\R}{{\mathbb R}}
\newcommand{\C}{{\mathbb C}}
\newcommand{\U}{{\operatorname{U}}}
\newcommand{\SO}{{\operatorname{SO}}}
\newcommand{\SU}{{\operatorname{SU}}}
\newcommand{\su}{{\operatorname{\mathfrak{su}}}}
\newcommand{\Spin}{{\operatorname{Spin}}}
\newcommand{\Hom}{{\operatorname{Hom}}}
\newcommand{\spann}{{\operatorname{span}}}
\newcommand{\spec}{{\operatorname{spec}}}
\newcommand{\Ric}{{\operatorname{Ric}}}
\newcommand{\Riem}{R}
\newcommand{\scal}{{\operatorname{scal}}}
\newcommand{\Id}{{\operatorname{Id}}}
\newcommand{\Tr}{{\operatorname{Tr}}}
\newcommand{\dimm}{{\operatorname{dim}}}
\newcommand{\dvol}{{\operatorname{\textit{d\/}\text{vol}}}}
\newcommand{\vol}{{\operatorname{vol}}}
\newcommand{\inv}{^{-1}}
\newcommand{\restr}[1]{\lower0.4ex\hbox{$|$}\lower0.7ex \hbox{$\scriptstyle{#1}$}}

\begin{document}

\title[On the Dirac spectrum of homogeneous 3-spheres]{On the Dirac spectrum of homogeneous 3-spheres}
\author{Jordi Kling}
\address{}
\email{\href{mailto:jordikling@posteo.de}{jordikling@posteo.de}}
\author{Dorothee Schueth}
\address{Institut f\"ur Mathematik, Humboldt-Universit\"at zu
Berlin, D-10099 Berlin, Germany}
\email{\href{mailto:schueth@math.hu-berlin.de}{schueth@math.hu-berlin.de}}

 
\begin{abstract}
We show that any two left-invariant metrics on $S^3\cong\SU(2)$ which are isospectral for the associated classical Dirac operator~$D$ must be isometric. In the case of left-invariant metrics of positive scalar curvature, we compute and use the smallest eigenvalue of~$D^2$.
We show analogous results for left-invariant metrics on $\SO(3)=S^3/\{\pm1\}$ for each of its two spin structures.
\end{abstract}

\maketitle
\thispagestyle{empty}

\Section{Introduction}

\noindent
On any closed Riemannian spin manifold, the associated classical Dirac operator~$D$ is an elliptic, self-adjoint operator with a discrete spectrum consisting of real eigenvalues of finite multiplicity. Just as in the case of the Laplace operator or any other geometric operator,
it is an interesting question in inverse spectral geometry to which extent the geometry of the underlying Riemannian manifold
is determined by the spectrum of~$D$.

There do exist many examples of Dirac isospectral, nonisometric manifolds. For example, any two flat tori which are nonisometric, but isospectral
for the Laplace operator on functions (it is well known that there exist such pairs in any dimension ${}\ge4$) also are Dirac isospectral for the associated trivial spin structures. This follows from Friedrich's computation of the Dirac spectra of flat tori; see, e.g., \cite{Gin}, p.~29. 
Examples of nonisometric, Dirac isospectral lens spaces in dimensions $4k+3\ge7$ were systematically constructed by Boldt and Lauret in~\cite{BL}. See section~6.1 of~\cite{Gin} for more examples. 

Among the results in the converse direction -- i.e., determination of geometric properties by the Dirac spectrum -- is the computation of the first few heat invariants of $D^2$ by Dlubek and Friedrich in~\cite{DF}, which can also be recovered by Gilkey's later, more general formulas in~\cite{Gil} for the heat invariants of general Laplace type operators in vector bundles over closed Riemannian manifolds.
Another example of a converse result is a theorem by Boldt which says that any two Dirac isospectral $3$-dimensional lens spaces with fundamental group of prime order and metric induced by the standard metric on the sphere must be isometric, see~\cite{Bo}.

Just as in the case of the Laplace operator on functions, explicit computation of Dirac spectra is possible only for quite restricted classes
of Riemannian manifolds. These include, for example, certain symmetric spaces, certain Bieberbach manifolds and Heisenberg manifolds, spherical spaceforms, Berger spheres and certain discretes quotients of these; see~\cite{Gin}, p.~38 for more details and references.

In the present article, we focus on homogeneous metrics on the compact $3$-dimensional Lie group $S^3\cong\SU(2)$ and its quotient
$\SO(3)=S^3/\{\pm1\}$. Each homogeneous metric on these is left-invariant and is isometric to, resp.~induced by, a metric of the form $\gabc$, where $a,b,c>0$
are the inverses of the lengths of the standard basis elements of $T_1 S^3=\spann_\R\{\ii,\jj,\kk\}\subset\HH$; see Section~\ref{sec:prelims} for more details.
Here, explicit computation of the full Dirac spectra is, just as for the Laplace spectrum, possible only in the case of Berger metrics, that is, in the case of metrics $\gabc$ where at least two of the numbers $a,b,c$ coincide. B\"ar~\cite{Ba} computed
the Dirac spectra of Berger metrics $S^3$ and on lens spaces $S^3/\Z_N$. Even though computation of the full Dirac spectrum is not possible for general $a,b,c$, one can aim, at least, at computing the eigenvalue of smallest absolute value; its square is the so-called fundamental tone of~$D^2$.

Lauret~\cite{La} succeeded in doing this for the Laplace operator on functions, both on $(S^3,\gabc)$ and $(\SO(3),\gabc)$ (see also~\cite{BLP} for a generalization to certain other homogeneous spaces). Moreover, Lauret proved in~\cite{La}
that -- within the set of homogeneous metrics -- the smallest Laplace eigenvalue of $\gabc$ on $S^3$ or $\SO(3)$, together with the volume and the scalar curvature, determines the underlying metric $\gabc$ up to isometry.

Our motivation for the present article was to achieve similar results for the Dirac operator. After choice of an orientation,
there exist precisely one spin structure over $(S^3,\gabc)$ and precisely two spin structures over $(\SO(3),\gabc)$, and we
consider the associated Dirac operators. Note that, since we are in odd dimension, changing the orientation replaces these
Dirac operators by their negatives, which has no effect on statements about the absolute values of eigenvalues.

Appallingly, it turns out that there is no hope of giving a general
formula (in terms of the metric) for the smallest absolute value of Dirac eigenvalues of $(S^3,\gabc)$ if one does not impose any additional geometric restrictions.
In fact, B\"ar~\cite{Ba} showed for $S^3$ with its canonical orientation: If $(a,b,c)=(\frac 1T,1,1)$ then each $-\frac T2+2n$ ($n\in\N$)
is contained in the Dirac spectrum of $(S^3,\gabc)$. Now, if one considers the corresponding $1$-parameter families of metrics depending on~$T$ and lets $T$ tend to infinity,
more and more of these particular eigenvalues successively pass across the zero point. So, any formula for the smallest absolute value
of eigenvalues would have to contain infinitely many case distinctions.

There is a quite canonical condition which prevents such a behavior, namely, positive scalar curvature. In fact, the well-known Lichnerowicz Theorem (see, e.g., Theorem~8.8 in~\cite{LM}) implies that any eigenvalue $\lambda$ of the Dirac operator satisfies $\lambda^2\ge\frac14\min(\scal)$, where $\scal$ is the scalar curvature. In our case, since the metrics $\gabc$ are homogeneous, the scalar curvature
is a constant (depending on $a,b,c$). The assumption of positive scalar curvature thus implies that $\frac 12\sqrt{\scal}$ is a common positive
lower
bound for the absolute values $|\lambda|$. 
Moreover, the condition $\scal>0$ is canonical in the sense that the spectrum of the Dirac operator determines \emph{whether} it is satisfied. In fact, the heat invariants of the square of the Dirac operator determine the volume and the total scalar curvature and, thus, in our homogeneous case,
the value of the scalar curvature.

Imposing the condition $\scal>0$, we will indeed be able to obtain explicit formulas for the Dirac eigenvalue of smallest
absolute value. Our first main theorem is:
\begin{theorem}
\label{thm:first}
Let $D$ denote the Dirac operator associated
with $(S^3,\gabc)$ (endowed with either orientation), and let $D^{\alpha_0}$ and $D^{\alpha_1}$ denote the Dirac operators associated
with the trivial and the nontrivial spin structure, respectively, on the quotient manifold $(\SO(3),\gabc)$.
Let
$$C\coloneqq\frac12\left(\frac{ab}c+\frac{bc}a+\frac{ca}b\right)\text{\ \ and \ \ }\mu\coloneqq a+b+c-C.
$$
If $\scal>0$ then
\begin{align*}
\min\{|\lambda|\mid \lambda\in\spec(D)\}&=\mu>0,\\
\min\{|\lambda|\mid \lambda\in\spec(D^{\alpha_0})\}&=C\ge\mu,\\
\min\{|\lambda|\mid \lambda\in\spec(D^{\alpha_1})\}&=\mu.
\end{align*}
The smallest eigenvalue of the squared operator $D^2$ has multiplicity $4$ if $a=b=c$ and multiplicity~$2$ otherwise.
The smallest eigenvalue of $(D^{\alpha_0})^2$ always has multiplicity~$2$; the same holds for $(D^{\alpha_1})^2$.
\end{theorem}

In the proof we will make use of B\"ar's general approach for describing the Dirac operator on a homogeneous manifold (see~\cite{Ba}); at the
same time, in the case of $S^3$, we generalize his considerations from the special case of Berger $3$-spheres to our arbitrary
homogeneous metrics $\gabc$. Just as in~\cite{La}, the Gershgorin Circle Theorem will play a technical key role in our arguments.
 
We will apply Theorem~\ref{thm:first} to show that if $\scal>0$ then for each of the above
three Dirac operators, the corresponding eigenvalue spectrum determines the isometry class
of $\gabc$ within the class of homogeneous metrics on the underlying manifold. More precisely, we will
show that the smallest absolute value of eigenvalues, together with the volume and the scalar curvature, determines $a,b,c$ up to permutation.
While that smallest absolute value is not available in the case of \emph{nonpositive} scalar curvature, it turns out that spectral
determination of the isometry class within the class of homogeneous metrics does hold nevertheless if $\scal\le0$. Our second main theorem is, thus, regardless of the sign of the scalar curvature:

\begin{theorem}\
\label{thm:isometry}
\begin{itemize}
\item[(i)]
Within the class of homogeneous metrics on~$S^3$, the metric $\gabc$ is determined by the spectrum of~$D$ up to isometry.
\item[(ii)]
Within the class of homogeneous metrics on $\SO(3)$, the metric $\gabc$ is determined by the spectrum of~$D^{\alpha_0}$ up to isometry.
The same holds for the spectrum of~$D^{\alpha_1}$.
\end{itemize}
\end{theorem}

For proving this in the case of nonpositive curvature we will
use the third heat invariant $a_2$ from~\cite{DF} for the square of the Dirac operator, together with the volume and $\scal$ (which correspond to $a_0$ and $a_1$; see Section~\ref{sec:invar} for more details). Notably, that argument will actually \emph{need} the condition $\scal\le0$. Again, the choice of orientation plays no role here because the square of the Dirac operator is invariant
under a change of orientation. Note that the analog of Theorem~\ref{thm:isometry} for the Laplace operator was proved not only in~\cite{La} by Lauret, but independently also in~\cite{LSS} by Lin, Schmidt, and Sutton, who used the corresponding heat invariants $a_0$, $a_1$, $a_2$, $a_3$.

In the case of $S^3$ and $\scal>0$, the results of Theorem~\ref{thm:first} and Theorem~\ref{thm:isometry}
are contained in the first author's unpublished master thesis~\cite{Kl}.

This paper is organized as follows: In Section~\ref{sec:prelims} we will set up the required preliminaries, give formulas for the Dirac
operators on $(S^3,\gabc)$ and $(\SO(3),\gabc)$, and describe their restrictions to certain finite-dimensional invariant subspaces
using Frobenius reciprocity, following the approach in B\"ar's article~\cite{Ba}.
In Section~\ref{sec:smallest} we will prove Theorem~\ref{thm:first}.
Finally, Section~\ref{sec:invar} is devoted to the proof of Theorem~\ref{thm:isometry}.

\Section{Preliminaries}
\label{sec:prelims}

\noindent
Let $\HH=\spann\{1,\ii,\jj,\kk\}$ denote the skew-field of quaternions. We consider the Lie group
$$S^3=\{x\in\HH\mid x\bar x = 1\}\subset\HH,
$$
which is isomorphic to $\SU(2)$
via $z-\jj w\mapsto\left(\smallmatrix z&-\bar w\\ w&\hphantom{-}\bar z\endsmallmatrix\right)$ for $z,w\in\C$. Left-invariant Riemannian metrics on $S^3$ correspond to Euclidean inner products on $T_1 S^3=\spann\{\ii,\jj,\kk\}$. Let
$$
a,b,c>0.
$$
\begin{definition}
\label{def:gabc}
We let
\begin{equation*}
\begin{gathered}
\gabc\coloneqq \text{ the left-invariant metric on }S^3\text{ with orthonormal basis }\{X_1,X_2,X_3\},\\
\text{where }X_1\coloneqq a\ii, \ X_2\coloneqq b\jj, \ X_3\coloneqq c\kk.
\end{gathered}
\end{equation*}
We denote the induced left-invariant metric on the quotient Lie group $\SO(3)=S^3/\{\pm1\}$ by $\gabc$ again.
\end{definition}

\begin{remark}
\label{rem:permut}
It is well known that every left-invariant metric on~$S^3$ or $\SO(3)$ is isometric to a metric of the form $\gabc$ (see, e.g., \cite{Mi}).
Moreover, for any permutation~$\sigma$ of $\{a,b,c\}$, $g_{\sigma(a)\sigma(b)\sigma(c)}$ on~$S^3$
is isometric to $\gabc$ by an orientation-preserving isometry which also descends to $\SO(3)=S^3/\{\pm1\}$.
For example, with $p\coloneqq \frac{\jj+\kk}{\sqrt2}$, the map $S^3\ni x\mapsto pxp\inv\in S^3$ is an orientation-preserving automorphism sending the left-invariant vector fields $\ii,\jj,\kk$ to $-\ii,\kk,\jj$, respectively, thus giving an isometry between $(S^3,g_{abc})$ and $(S^3,g_{acb})$,
which obviously also descends to $\SO(3)$.
\end{remark}

Consider the Riemannian manifold $(S^3,\gabc)$, endowed with the standard orientation, with respect to which $\{\ii,\jj,\kk\}$ is positively oriented.
Up to isomorphism, there exists exactly one spin structure over this oriented Riemannian manifold, namely:
$$F:\Spin(S^3)\coloneqq S^3\times\Spin(3)\ni(x,u)\mapsto\{x\,a\ii, x\,b\jj, x\,c\kk\}\theta(u)\in\SO(S^3,\gabc)_x\subset\SO(S^3,\gabc),
$$
where $\theta:\Spin(3)\to\SO(3)$ is the usual two-fold covering and $\SO(3)$ acts canonically from the right on the bundle $\SO(S^3,\gabc)$ of positively oriented orthonormal bases of $(S^3,\gabc)$.

In \cite{Ba}, C.~B\"ar established formulas for spin structures, spinor bundles, and the associated Dirac operators over an oriented homogeneous space $M=G/H$, endowed with a $G$-invariant Riemannian metric. We specialize these to the case $M=G/H=S^3/\{1\}\cong S^3$, endowed with the above orientation. Then the above spin structure $S^3\times\Spin(3)$ corresponds to B\"ar's $G\times_{\alpha'}\Spin(n)$ with $n=3$ and
$$\alpha':H=\{1\}\ni 1\mapsto 1\in\Spin(3).
$$
Consequently (see Lemma~4 of \cite{Ba} in the special case $\alpha'=1$), the associated spinor bundle is isomorphic to
$$\Sigma S^3=S^3\times\Sigma_3,
$$
where
$$\Sigma_3\cong\C^2
$$
denotes the space of complex $3$-spinors. In particular, smooth sections of $\Sigma S^3$ can be viewed
as smooth maps $\phi:S^3\to\Sigma_3$. Applying Theorem~1 of~\cite{Ba} in our special case, we easily obtain the following
formula for the associated Dirac operator~$D$:

\begin{remark}
\begin{equation}
\label{eq:dirac}
(D\phi)(x)=\sum_{\ell=1}^3 e_\ell\cdot X_\ell\restr{x}(\phi)+\frac12\left(\frac{ab}c+\frac{bc}a+\frac{ca}b\right)e_1\cdot e_2\cdot
e_3\cdot\phi(x)
\end{equation}
for all $\phi\in C^\infty(S^3,\Sigma_3)$ and $x\in S^3$, where $\mathbf{\cdot}$ denotes Clifford multiplication, $\{X_1,X_2,X_3\}$ is the above orthonormal basis $\{a\ii,b\jj,c\kk\}$ regarded as left-invariant vector fields,
and $\{e_1,e_2,e_3\}$ denotes the standard basis of~$\R^3$. (Note that we have $[X_1,X_2]=\frac{2ab}c X_3$ and the analogous formulas arising from cyclic permutation, which implies, in B\"ar's notation, $\beta_i=0$ for $i=1,2,3$, and $\alpha_{123}=
\frac12\left(\frac{ab}c+\frac{bc}a+\frac{ca}b\right)$.)
\end{remark}

Let $\pi_n$ denote the irreducible representation of $S^3\cong\SU(2)$ on the $(n+1)$-dimensional complex vector space
$V_n\coloneqq \spann\{P_k\mid k=0,...,n\}$ of the polynomials $P_k(z,w)=z^{n-k}w^k$, given by $(\pi_n(x)P)(z,w)=P((z,w)x)$.
The group $S^3$ acts on $L^2(S^3,\Sigma_3)$ by $(x\phi)(y)\coloneqq \phi(x\inv y)$.
By Frobenius reciprocity, the decomposition of this representation space into isotypical components is given by
\begin{equation}
\label{eq:frob}
L^2(S^3,\Sigma_3)\cong\overline{\bigoplus_{n\in\N_0}V_n\otimes\Hom(V_n,\Sigma_3)},
\end{equation}
where $v\otimes f$ in $V_n\otimes\Hom(V_n,\Sigma_3)$ correponds to $f(\pi_n(\,.\,)\inv v)\in C^\infty(S^3,\Sigma_3)$.

The result of applying $X_\ell\restr{x}$ to the latter is $f(\pi_{n*}(-X_\ell)\pi_n(x)\inv v)$ and corresponds to the vector
$v\otimes(f\circ\pi_{n*}(-X_\ell))$. Thus, by~\eqref{eq:dirac} we get the following special case of Proposition~1 of~\cite{Ba}:

\begin{proposition}
\label{prop:diracvn}
The restriction of $D$ to the isotypical component $V_n\otimes\Hom(V_n,\Sigma_3)$ of the above decomposition
of $L^2(S^3,\Sigma_3)$ is given by $\Id\otimes D_n$, where
$$D_n(f)=-\sum_{\ell=1}^3 e_\ell\cdot f\circ\pi_{n*}(X_\ell)+\frac12\left(\frac{ab}c+\frac{bc}a+\frac{ca}b\right)e_1\cdot e_2\cdot
e_3\cdot f
$$
for all $f\in\Hom(V_n,\Sigma_3)$.
\end{proposition}

In the above formula, $X_\ell\in T_e\SU(2)=\su(2)$ has to be interpreted as the vector corresponding to $X_\ell\in T_1 S^3$ under the identification given at the beginning of this section. That is:
$$X_1=a\left(\begin{matrix}i&\hphantom{-}0\\0&-i\end{matrix}\right), \
X_2=b\left(\begin{matrix}\hphantom{-}0&1\\-1&0\end{matrix}\right), \
X_3=c\left(\begin{matrix}0&i\\i&0\end{matrix}\right)
$$

\begin{corollary}
\label{cor:speccollection}
For the eigenvalue spectrum $\spec(D)$ of $D$ we have
$$\spec(D)=\{\lambda\in\R\mid\lambda\text{ is an eigenvalue of }D_n\text{ for some }n\in\N_0\},
$$
and if $\lambda$ is an eigenvalue of $D_n$ with multiplicity~$m$, then the contribution of~$D_n$
to the total multiplicity of $\lambda$ as an eigenvalue of~$D$ is $\;m\,\dimm(V_n)=m(n+1)$.
\end{corollary}

\begin{remark}
Let $\SO(3)=S^3/\{\pm1\}$ be endowed with the orientation and Riemannian metric $\gabc$ inherited from~$S^3$.
There are precisely two spin structures over $(\SO(3),\gabc)$, corresponding to the two homomorphisms
$$\alpha_0:\{\pm1\}\ni\gamma\mapsto 1\in\Spin(3)\ \ \text{and}\ \ \alpha_1:\{\pm1\}\ni\gamma\mapsto\gamma\in\Spin(3).
$$
According to Lemma~4 of~\cite{Ba},
the associated spinor bundles are isomorphic to
$$\Sigma_{\alpha_j}\SO(3)\coloneqq S^3\times_{\rho\circ\alpha_j}\Sigma_3,
$$
for $j\in\{0,1\}$, where $\rho:\Spin(3)\to\U(\Sigma_3)$ is the spin representation.
In particular, smooth sections of $\Sigma_{\alpha_j}\SO(3)$ can be viewed as those smooth maps $\phi:S^3\to\Sigma_3$ which
are equivariant with respect to $\rho\circ\alpha_j$. Note that $(\rho\circ\alpha_0)(-1)=\Id_{\Sigma_3}$ and
$(\rho\circ\alpha_1)(-1)=-\Id_{\Sigma_3}$. Thus, smooth sections of $\Sigma_{\alpha_0}\SO(3)$ or $\Sigma_{\alpha_1}\SO(3)$
can be viewed as those smooth maps $\phi:S^3\to\Sigma_3$ which are $\{\pm1\}$ invariant or anti-invariant, respectively.
Moreover, Theorem~1 of~\cite{Ba} implies that the corresponding
Dirac operators are just the restrictions of~$D$ to these subspaces of sections.
Since we have $\pi_n(-x)=(-1)^n\pi_n(x)$ for all $x\in S^3$ and $n\in\N_0$, where $\pi_n$ are the above irreducible representations
of $S^3\cong\SU(2)$, we conclude that the sections of $\Sigma_{\alpha_0}\SO(3)$ or $\Sigma_{\alpha_1}\SO(3)$ correspond to those
summands in~\eqref{eq:frob} where $n$ is even or odd, respectively.  Thus, the following corollary follows from Corollary~\ref{cor:speccollection}.
\end{remark}

\begin{corollary}
\label{cor:speccollectionquot}
For $j\in\{0,1\}$ let $D^{\alpha_j}$ denote the Dirac operator on $\Sigma_{\alpha_j}\SO(3)$. 
For the corresponding eigenvalue spectra we then have
\begin{align*}
\spec(D^{\alpha_0})&=\{\lambda\in\R\mid\lambda\text{ is an eigenvalue of }D_n\text{ for some even }n\in\N_0\},\\
\spec(D^{\alpha_1})&=\{\lambda\in\R\mid\lambda\text{ is an eigenvalue of }D_n\text{ for some odd }n\in\N_0\},
\end{align*}
with multiplicities as described in Corollary~\ref{cor:speccollection}.
\end{corollary}

\begin{remark}
\label{rem:matrices}

(i)
With respect to the basis $\{P_0,\ldots,P_n\}$ of $V_n$, the endomorphisms $\pi_{n*}(X_\ell)$ are given as follows (where terms involving
``$P_{-1}$'' or ``$P_{n+1}$'' have to be interpreted as zero):
\begin{equation*}
\begin{aligned}
\pi_{n*}(X_1)(P_k)&=ai(n-2k)P_k\\
\pi_{n*}(X_2)(P_k)&=bkP_{k-1}-b(n-k)P_{k+1}\\
\pi_{n*}(X_3)(P_k)&=cikP_{k-1}+ci(n-k)P_{k+1}
\end{aligned}
\end{equation*}
for $k=0,\ldots,n$.

\vskip5pt
(ii)
We choose a basis $\{Z_1,Z_2\}$ of $\Sigma_3\cong\C^2$ with respect to which
Clifford multiplication by the standard basis vectors $e_1,e_2,e_3\in\R^3$ is
given by
$$\left(\begin{smallmatrix}i&\hphantom{-}0\\0&-i\end{smallmatrix}\right), \
\left(\begin{smallmatrix}\hphantom{-}0&1\\-1&0\end{smallmatrix}\right), \
\left(\begin{smallmatrix}0&i\\i&0\end{smallmatrix}\right),
$$
respectively. (In particular, $e_1\cdot e_2\cdot e_3\cdot Z=-Z$ for all $Z\in\Sigma_3$.)

\vskip5pt
(iii)
In B\"ar's notation in Section~4 of~\cite{Ba}, the latter three matrices correspond to the Clifford multiplications of $e_3,-e_2,e_1$, respectively. In view of the formula from Proposition~\ref{prop:diracvn}, this harmonizes with the fact that our $X_1,X_2,X_3$, too, correspond to B\"ar's $X_3,-X_2,X_1$ up to scaling. Note that the case of Berger spheres treated in Sections~4 to~6 of~\cite{Ba} corresponds to $a=\frac1T$ and $b=c=1$ in our notation. 
\end{remark}

\begin{definition}
\label{def:CAB}

(i)
For any given $n$, we consider, as in~\cite{Ba}, the basis $\{A_0,\ldots,A_n,B_1,\ldots,B_n\}$ of $\Hom(V_n,\Sigma_3)$ given by
\begin{equation*}
A_k(P_m)=\begin{cases}Z_1,&k=m,\text{ $k$ even,}\\Z_2,&k=m,\text{ $k$ odd,}\\0,&\text{otherwise,}\end{cases}\text{\ \ \ \ }
B_k(P_m)=\begin{cases}Z_1,&k=m,\text{ $k$ odd,}\\Z_2,&k=m,\text{ $k$ even,}\\0,&\text{otherwise,}\end{cases}
\end{equation*}
for $k,m\in\{0,\ldots,n\}$.

\vskip5pt
(ii)
We abbreviate
\begin{equation}
\label{eq:defC}
C\coloneqq \frac12\left(\frac{ab}c+\frac{bc}a+\frac{ca}b\right).
\end{equation}
\end{definition}

As it turns out, the two subspaces of $\Hom(V_n,\Sigma_3)$ spanned by $\{A_0,\ldots,A_n\}$ and $\{B_0,\ldots,B_n\}$, respectively,
are invariant under the above operators $D_n$. More precisely,
using the above definitions and Remark~\ref{rem:matrices}(i),~(ii), one easily obtains the following generalization
of B\"ar's formulas on p.~74 of~\cite{Ba} (which correspond to the case $a=\frac1T$ and $b=c=1$):

\begin{corollary}
\label{cor:dnexpl}
$$D_n=D'_n-C\,\Id,
$$
where $D'_n$ is given by
\begin{equation*}
\begin{aligned}
D'_n(A_k)&=\begin{cases}(c-b)(n-k+1)A_{k-1}+a(n-2k)A_k+(c+b)(k+1)A_{k+1},&\text{ $k$ even},\\
(c+b)(n-k+1)A_{k-1}-a(n-2k)A_k+(c-b)(k+1)A_{k+1},&\text{ $k$ odd},\end{cases}\\
D'_n(B_k)&=\begin{cases}
(c+b)(n-k+1)B_{k-1}-a(n-2k)B_k+(c-b)(k+1)B_{k+1},&\text{ $k$ even},\\
(c-b)(n-k+1)B_{k-1}+a(n-2k)B_k+(c+b)(k+1)B_{k+1},&\text{ $k$ odd},\end{cases}
\end{aligned}
\end{equation*}
for $k\in\{0,\ldots,n\}$.
\end{corollary}

Of course, terms involving ``$A_{-1}$'', ``$A_{n+1}$'', ``$B_{-1}$'', or ``$B_{n+1}$'' are to be interpreted as zero in the previous corollary. In the next section, we will determine the eigenvalues of $D$, $D^{\alpha_0}$, $D^{\alpha_1}$ of smallest absolute value under the assumption that the scalar curvature associated with $\gabc$ is strictly positive. To that end, we will need:

\begin{proposition}
\label{prop:scal}
The scalar curvature associated with $\gabc$ on $(S^3,\gabc)$ or $(\SO(3),\gabc)$ is
\begin{align*}
\scal&=4(a^2+b^2+c^2)-2\left(\frac{a^2b^2}{c^2}+\frac{b^2c^2}{a^2}+\frac{c^2a^2}{b^2}\right)
       =8(a^2+b^2+c^2-C^2).
\end{align*}
\end{proposition}

The first equality follows from Milnor's formulas on p.\;305/306 of~\cite{Mi} (observing that his $\lambda_1,\lambda_2,\lambda_3$ are $\frac{2bc}a, \frac{2ca}b, \frac{2ab}c$, respectively, in our notation). The second equality follows easily using the definition of~$C$. Note that the scalar curvature is a constant function on $S^3\cong\SU(2)$ and on its Riemannian quotient $(\SO(3),\gabc)=(S^3,\gabc)/\{\pm1\}$ because $\gabc$ is homogeneous.

\Section{The smallest Dirac eigenvalues of \texorpdfstring{$(S^3,\gabc)$}{(S3,gabc)} and \texorpdfstring{$(\SO(3),\gabc)$}{(SO(3),gabc)}}
\label{sec:smallest}

\noindent
This section is devoted to proving Theorem~\ref{thm:first}. As explained in the introduction, it is no loss of generality if we continue
to consider only the standard orientation of $S^3$ and $\SO(3)$ because switching it would just change the sign of the Dirac operators.
Recall that
$$\mu\coloneqq a+b+c-C,
$$
where~$C$ was defined in~\eqref{eq:defC}.
Actually, we will prove the following more explicit result, which, in view of Corollary~\ref{cor:speccollection} and
Corollary~\ref{cor:speccollectionquot}, will imply Theorem~\ref{thm:first}:

\begin{theorem}
\label{thm:firstexplicit}
Assume $\scal>0$. Then we have
\begin{itemize}
\item[(i)] $-C$ is the only eigenvalue of~$D_0$, its multiplicity is~$2$, and\\
$C<\min\{|\lambda|\mid\lambda\in\spec(D_n)\}$ for all even $n\ge2$,
\item[(ii)] $\mu$ is an eigenvalue of~$D_1$ of multiplicity~$1$, $\mu=\min\{|\lambda|\mid\lambda\in\spec(D_1)\}$, and\\
$\mu<\min\{|\lambda|\mid\lambda\in\spec(D_n)\}$ for all odd $n\ge3$.
\end{itemize}
Moreover, $C\ge\mu>0$; $C=\mu$ holds if and only if $a=b=c$.
\end{theorem}

We start with some elementary estimates:

\begin{lemma}\
\label{lem:bounds}
\begin{itemize}
\item[(i)] $C^2\ge\max\{a^2+b^2,b^2+c^2,c^2+a^2\}$. In particular, $C>\max\{a,b,c\}$.
\item[(ii)] $2C\ge a+b+c$; in particular, $C\ge\mu$. Equality holds if and only if $a=b=c$.
\item[(iii)] $\scal>0$ is equivalent to $C^2<a^2+b^2+c^2$.\\
In particular, $\scal>0$ implies $C<a+b+c$; that is, $\mu>0$.
\item[(iv)] $\scal>0$ is equivalent to the following being simultaneously satisfied:\\
$ab<c(a+b)$ and $bc<a(b+c)$ and $ca<b(c+a)$.
\item[(v)] $\scal>0$ is equivalent to
$C<\min\{a+\frac{bc}a,\,b+\frac{ca}b,\,c+\frac{ab}c\}$.
\end{itemize}
\end{lemma}

\begin{proof}
(i)
$C^2\ge C^2-\frac14(-\frac{ab}c+\frac{bc}a+\frac{ca}b)^2=\frac14(\frac{ab}c+\frac{bc}a+\frac{ca}b)^2-\frac14(-\frac{ab}c+\frac{bc}a+\frac{ca}b)^2=a^2+b^2$.
Similarly, $C^2\ge b^2+c^2$ and $C^2\ge c^2+a^2$.

\vskip5pt
(ii)
\begin{equation*}
\begin{aligned}(ab)^2+(bc)^2+(ca)^2&=\tfrac12((ab)^2+(bc)^2)+\tfrac12((bc)^2+(ca)^2)+\tfrac12((ca)^2+(ab)^2)\\
&\ge ab\cdot bc+bc\cdot ca+ca\cdot ab=abc(a+b+c).
\end{aligned}
\end{equation*}
The claimed estimate follows by dividing by $abc$, and equality holds if and only if $ab=bc=ca$, that is, $a=b=c$.

\vskip5pt
(iii)
This is clear from Proposition~\ref{prop:scal}.

\vskip5pt
(iv)
Using Proposition~\ref{prop:scal} one easily verifies that
\begin{equation}
\label{eq:scalfactor}
\scal=\tfrac2{a^2b^2c^2}(ab+bc+ca)(ab+bc-ca)(ab-bc+ca)(-ab+bc+ca).
\end{equation}
It is not possible that more than one of the last three factors is negative because the sum of any two such factors is positive.
Thus, $\scal>0$ is equivalent to all factors being positive.

\vskip5pt
(v)
{}From the previous observation, $\scal>0$ is also equivalent to any two of the last three factors in~\eqref{eq:scalfactor}
having the same
sign, that is, $(ab+bc-ca)(ab-bc+ca)>0$ and the two analogous inequalities obtained by cyclic permutations. The statement now
follows from
$$(ab+bc-ca)(ab-bc+ca)=a^2b^2-(bc-ca)^2=abc(\tfrac{ab}c-\tfrac{bc}a-\tfrac{ca}b+2c)=2abc(-C+c+\tfrac{ab}c)
$$
and the two analogous equalities obtained by cyclic permutations.
\end{proof}

In order to prove Theorem~\ref{thm:firstexplicit}, we will use an inductive argument using the Gershgorin Circle Theorem applied to the operators $D_n^2$. That argument, however, will not work for small~$n$.
Therefore, we first treat the individual cases $n=0,2,4$ and $n=1,3$.

\begin{remark}
\label{rem:smalln}

In this remark, we always assume
$$
\scal>0.
$$
We use the explicit formulas for $D_n=D'_n-C\Id$ from Corollary~\ref{cor:dnexpl}, and we denote by $\mathcal{A}'_n$ and $\mathcal{B}'_n$ the matrices expressing the restrictions of $D'_n$ to the invariant subspaces $\spann\{A_0,\ldots,A_n\}$ and $\spann\{B_0,\ldots,B_n\}$, respectively.

\vskip5pt
(i)
$\mathbf{n=0}$: We have $D'_0=0$, hence $D_0=-C\Id$. In particular, $-C$ is the only eigenvalue of $D_0$. The dimension of
$\Hom(V_0,\Sigma_3)\cong\Hom(\C,\C^2)$ on which $D_0$ lives is $2$, so $-C$ has multiplicity~$2$ as an eigenvalue of~$D_0$.

\vskip5pt
(ii)
$\mathbf{n=2}$:
One easily calculates that the characteristic polynomials of $\mathcal{A}'_2$ and $\mathcal{B}'_2$ coincide and are given by
$$\chi_2(x)\coloneqq x^3-4(a^2+b^2+c^2)x-16abc.
$$
Note that $\chi_2$ is strictly convex on $[0,\infty)$ since $\chi_2''(x)=6x$. Moreover, $\chi_2(0)=-16abc<0$ and
$\chi_2(2C)=8C^3-8(a^2+b^2+c^2)C-16abc=-8C(a^2+b^2+c^2-C^2)-16abc<0$ by Lemma~\ref{lem:bounds}(iii). Thus, $\chi_2$ is strictly negative on $[0,2C]$, so $D'_2$ has no eigenvalues in $[0,2C]$. It follows that $D_2$ has no eigenvalues in $[-C,C]$.

\vskip5pt
(iii)
$\mathbf{n=4}$:
The characteristic polynomials of $\mathcal{A}'_4$ and $\mathcal{B}'_4$ coincide and are given by
\begin{equation*}
\begin{split}
\chi_4(x)\coloneqq x^5&-20(a^2+b^2+c^2)x^3-80abc\,x^2\\&+64(a^4+b^4+c^4+4a^2b^2+4b^2c^2+4c^2a^2)x+768abc(a^2+b^2+c^2).
\end{split}
\end{equation*}
We will show that this is strictly positive on $[0,2C]$, so $D_4$ has no eigenvalues in $[-C,C]$. Note that
$$\chi_4''(x)=20x^3-120(a^2+b^2+c^2)x-160abc
$$
is strictly convex on $[0,\infty)$ since its second derivative is $120x$. Moreover, $\chi_4''(0)=-160abc<0$
and
\begin{equation*}
\chi_4''(2C)=160 C^3-240(a^2+b^2+c^2)C-160abc<0
\end{equation*}
since $a^2+b^2+c^2>C^2$ by Lemma~\ref{lem:bounds}(iii).
It follows that $\chi_4''<0$ on $[0,2C]$, so $\chi_4$ is concave on this interval.
Moreover, $\chi_4(0)=768abc(a^2+b^2+c^2)>0$ and, again using $C^2<a^2+b^2+c^2$:
\begin{equation*}
\begin{split}
\chi_4(2C)&=32C^5-160(a^2+b^2+c^2)C^3-320abc\,C^2\\&\ \ +128(a^4+b^4+c^4+4a^2b^2+4b^2c^2+4c^2a^2)C+768abc(a^2+b^2+c^2)\\
&>32C^5-160(a^2+b^2+c^2)^2C+128(a^4+b^4+c^4+4a^2b^2+4b^2c^2+4c^2a^2)C.
\end{split}
\end{equation*}
Assume for a moment that $b \ge c$, and recall that $C^2>a^2+b^2$ by Lemma~\ref{lem:bounds}(i). Then
\begin{equation*}
\begin{split}
\chi_4(2C)&>32C((a^2+b^2)^2+(-5+4)(a^2+b^2+c^2)^2+8(a^2b^2+b^2c^2+c^2a^2))\\
&=32C(-c^4+6a^2c^2+6b^2c^2+8a^2b^2)>0
\end{split}
\end{equation*}
since $-c^4\ge -b^2c^2$. The case $b\le c$ is treated analogously. By concavity of $\chi_4$ on $[0,2C]$
and by $\chi_4(0)>0$ it follows that $\chi_4$ is strictly positive on $[0,2C]$, as claimed above.

\vskip5pt
(iv)
$\mathbf{n=1}$: We have $\mathcal{A}'_1=\left(\begin{smallmatrix} a&c+b \\c+b&a \end{smallmatrix}\right)$,
 $\mathcal{B}'_1=\left(\begin{smallmatrix} -a&c-b \\c-b&-a \end{smallmatrix}\right)$. Thus, $D_1=D'_1-C\Id$ has the eigenvalues
$$a+b+c-C,\ \ a-b-c-C,\ \ -a+b-c-C,\ \ -a-b+c-C.
$$
Note that the first of these equals~$\mu$. The second is $a-b-c-C<-a-b-c+C=-\mu<0$ by Lemma~\ref{lem:bounds}(i), (iii). In particular, $|a-b-c-C|>\mu$; this follows analogously for the third and fourth eigenvalue, too. Thus, $\mu$ has multiplicity~$1$ as an eigenvalue of~$D_1$. Recall that
$D_1$ contributes each of its eigenvalues twice to the spectrum of~$D$, resp.~$D^{\alpha_1}$ by Corollary~\ref{cor:speccollection}, resp.
Corollary~\ref{cor:speccollectionquot}.

\vskip5pt
(v)
$\mathbf{n=3}$:
It turns out that the eigenvalues of the $4\times 4$-matrices $\mathcal{A}'_3$ and $\mathcal{B}'_3$ can be computed explicitly and are given as follows, where $\fbox{1}$ -- $\fbox{4}$ belong to $\mathcal{A}'_3$ and $\fbox{5}$ -- $\fbox{8}$ to $\mathcal{B}'_3$:
\begin{equation*}
\begin{aligned}
&\fbox{1}\ a+b-c-2\sqrt{a^2+b^2+c^2-ab+bc+ca}\ &\fbox{5}\ -a-b-c-2\sqrt{a^2+b^2+c^2-ab-bc-ca}\\
&\fbox{2}\ a+b-c+2\sqrt{a^2+b^2+c^2-ab+bc+ca}\ &\fbox{6}\ -a-b-c+2\sqrt{a^2+b^2+c^2-ab-bc-ca}\\
&\fbox{3}\ a-b+c-2\sqrt{a^2+b^2+c^2+ab+bc-ca}\ &\fbox{7}\ -a+b+c-2\sqrt{a^2+b^2+c^2+ab-bc+ca}\\
&\fbox{4}\ a-b+c+2\sqrt{a^2+b^2+c^2+ab+bc-ca}\ &\fbox{8}\ -a+b+c+2\sqrt{a^2+b^2+c^2+ab-bc+ca}
\end{aligned}
\end{equation*}
We want to show that none of the above lies in $[C-\mu,C+\mu]=[2C-a-b-c,a+b+c]$; this will imply that $D_3$ has no
eigenvalues in $[-\mu,\mu]$. For $\fbox{1}$, note that
$$U\coloneqq a^2+b^2+c^2-ab+bc+ca=\tfrac12((a-b)^2+(b+c)^2+(c+a)^2)>\tfrac12(0+b^2+a^2)\ge\min\{a^2,b^2\}.
$$
Thus, $C-a-b+\sqrt{U}>C-a-b+\min\{a,b\}=C-\max\{a,b\}>0$ by Lemma~\ref{lem:bounds}(i). Hence, the expression in~$\fbox{1}$ is strictly smaller than $2C-a-b-c$.
Similarly, one shows the same for the expressions in \fbox{3} and \fbox{7}. Trivially, the expression in \fbox{5} is negative, hence strictly smaller than $C-\mu$.
For \fbox{2}, note that $U>\frac12(0+c^2+c^2)=c^2$, hence $a+b-c+2\sqrt{U}>a+b+c$. Analogously, the same holds for the expressions in
\fbox{4} and~\fbox{8}. Concerning~\fbox{6}, assume for a moment that $a\ge b\ge c$. Then $a^2+b^2+c^2-ab-bc-ca\le a^2+b^2+c^2-b^2-c^2-c^2
=a^2-c^2<a^2<C^2$ by Lemma~\ref{lem:bounds}(i). By symmetry in $a,b,c$, the same holds for permutations. In particular, the expression in~\fbox{6} is strictly smaller than $-a-b-c+2C$.
\end{remark}

We will make use of the classical Gershgorin Circle Theorem in its version for the rows of a matrix:

\begin{theorem}[Gershgorin Circle Theorem]
If $M=(m_{k,\ell})$ is a complex square matrix then each of its eigenvalues is contained in at least one of the closed discs
centered at $m_{k,k}$ with radius $\sum_{\ell\ne k}|m_{k,\ell}|$.
\end{theorem}

We will apply this to the squares of the matrices
$$\mathcal{A}_n\coloneqq \mathcal{A}'_n-C\,I\text{ \ and \ }\mathcal{B}_n\coloneqq \mathcal{B}'_n-C\,I,
$$
where $I$ denotes the identity matrix in dimension $n+1$. Recall that $\mathcal{A}_n$ and $\mathcal{B}_n$ are the blocks
of the matrix representation of $D_n$ with respect to the basis $\{A_0,\ldots,A_n,B_0,\ldots,B_n\}$ of
$\Hom(V_n,\Sigma_3)$ from Definition~\ref{def:CAB}. Since $\mathcal{A}_n$ and $\mathcal{B}_n$ are tridiagonal, their squares
are pentadiagonal, and the following corollary is immediate.

\begin{corollary}
\label{cor:minmin}
Let the rows and columns of $\mathcal{A}_n$ and $\mathcal{B}_n$ be numbered by $0,\ldots,n$. Then for
each eigenvalue $\lambda$ of $D_n^2$ we have
$$\lambda\ge\min\bigl\{\min\{G_{\mathcal{A}}(n,k)\mid k=0,\ldots,n\},
\min\{G_{\mathcal{B}}(n,k)\mid k=0,\ldots,n\}\bigr\},
$$
where
$$\smash[t]{G_{\mathcal{A}}(n,k)\coloneqq \mathcal({A}_n^2)_{k,k}-\sum_{\ell=k-2}^{k+2}|(\mathcal{A}_n^2)_{k,\ell}|\text{ \ and \ }
G_{\mathcal{B}}(n,k)\coloneqq \mathcal({B}_n^2)_{k,k}-\sum_{\ell=k-2}^{k+2}|(\mathcal{B}_n^2)_{k,\ell}|}
$$
are the left endpoints of the $k$-th Gershgorin intervals corresponding to $\mathcal{A}_n^2$, resp.~$\mathcal{B}_n^2$.
(Entries $(\mathcal{A}_n^2)_{k,\ell}$, $(\mathcal{B}_n^2)_{k,\ell}$ involving $\ell<0$ or $\ell>n$ are to be interpreted as zero.)
\end{corollary}

Using Corollary~\ref{cor:dnexpl} one directly computes the following formulas for the entries of the rows of $\mathcal{A}_n^2$
and $\mathcal{B}_n^2$:

\begin{proposition}
\label{prop:sqentries}
If $k\in\{0,\ldots,n\}$ is even then
\begin{align*}
(\mathcal{A}_n^2)_{k,k-2}&=(c-b)(c+b)k(k-1)\\
(\mathcal{A}_n^2)_{k,k-1}&=-2(c-b)(C+a)k\\
(\mathcal{A}_n^2)_{k,k}&=(c-b)^2k(n-k+1)+(a(n-2k)-C)^2+(c+b)^2(n-k)(k+1)\\
(\mathcal{A}_n^2)_{k,k+1}&=-2(c+b)(C-a)(n-k)\\
(\mathcal{A}_n^2)_{k,k+2}&=(c+b)(c-b)(n-k)(n-k-1)
\end{align*}
If $k\in\{0,\ldots,n\}$ is odd, then the same formulas hold with $a,b$ replaced by $-a,-b$, respectively.
The entries of $\mathcal{B}_n^2$ are obtained by swapping ``$k$ even'' and ``$k$ odd''.
\end{proposition}

In the above proposition, note that entries which have to be interpreted as zero
(like, e.g., $(\mathcal{A}_n^2)_{0,-2}$, $(\mathcal{A}_n^2)_{n,n+1}$) are described by formulas which indeed happen to be zero.

\begin{corollary}
\label{cor:G01}
If we assume
$$b\ge c
$$
then
\begin{equation*}
G_\mathcal{A}(n,k)=\begin{cases}G(n,k),&k\text{ even},\\
\tilde G(n,k),&k\text{ odd,}\end{cases}
\text{\ \  \ and\ \ \ }
G_\mathcal{B}(n,k)=\begin{cases}\tilde G(n,k),&k\text{ even},\\
G(n,k),&k\text{ odd,}\end{cases}
\end{equation*}
where
\begin{align*}
G(n,k)&\coloneqq (a(n-2k)-C)^2+(b-c)^2k(n-k+1)+(b+c)^2(n-k)(k+1)\\
&\ \ -2(b-c)(C+a)k-2(b+c)(C-a)(n-k)\\
&\ \ -(b^2-c^2)(k(k-1)+(n-k)(n-k-1))
\end{align*}
and
\begin{align*}
\tilde G(n,k)&\coloneqq (a(n-2k)+C)^2+(b+c)^2k(n-k+1)+(b-c)^2(n-k)(k+1)\\
&\ \ -2(b+c)(C-a)k-2(b-c)(C+a)(n-k)\\
&\ \ -(b^2-c^2)(k(k-1)+(n-k)(n-k-1)).
\end{align*}
Moreover, for all $n\in\N_0$ and $k\in\{0,\ldots,n\}$ we have
$$G(n,k)=\tilde G(n,n-k).
$$
\end{corollary}

\begin{proof}
Using the assumption $b\ge c$ and recalling that $C>a$ by Lemma~\ref{lem:bounds}(i), we have the signs $(-,+,*,-,-)$ in the $k$-th row of $\mathcal{A}_n^2$ from Proposition~\ref{prop:sqentries} 
if $k$ is even and $(-,-,*,+,-)$ if $k$ is odd. This immediately implies the formulas for $G_\mathcal{A}(n,k)$; recall the observation that the terms corresponding to nonexistent entries
like $(\mathcal{A}_n^2)_{0,-2}$ etc.~indeed evaluate to zero here. Again, $G_\mathcal{B}$ is obtained from $G_\mathcal{A}$ by swapping
``$k$ even'' and ``$k$ odd''.
The last statement of the corollary is obvious.
\end{proof}

\begin{remark}
\label{rem:gershgoals}
Recall from Remark~\ref{rem:smalln} that we already proved the statements of Theorem~\ref{thm:firstexplicit} for $0\le n\le4$. In order to prove Theorem~\ref{thm:firstexplicit}, it will suffice to prove the following:

If $\scal>0$ and $a\ge b\ge c$ then
\begin{align}
G(n,k)&> C^2 \textit{ for all even $n\ge6$ and all }k\in\{0,\ldots,n\},\label{eq:neven}\\
G(n,k)&>\mu^2 \textit{ for all odd $n\ge5$ and all }k\in\{0,\ldots,n\}.\label{eq:nodd}
\end{align}
In fact, in view of Corollaries~\ref{cor:minmin} and~\ref{cor:G01}, this will imply the statements of Theorem~\ref{thm:firstexplicit} for
all $n\ge 5$ in case $a\ge b\ge c$. The condition $a\ge b\ge c$ can then be removed because of Remark~\ref{rem:permut}.
\end{remark}

The key observation for proving \eqref{eq:neven} and \eqref{eq:nodd} by induction is the next lemma. It says that the values
in each column of the following two triangular diagrams are strictly increasing if $\scal>0$:

\begin{table}[h]
\centering
\begin{tabular}{ccccccccccc}
&&&&&G(0,0)&&&&&\\
&&&&G(2,0)&G(2,1)&G(2,2)&&&&\\
&&&G(4,0)&G(4,1)&G(4,2)&G(4,3)&G(4,4)&&&\\
&&G(6,0)&G(6,1)&G(6,2)&G(6,3)&G(6,4)&G(6,5)&G(6,6)&&\\
&G(8,0)&G(8,1)&G(8,2)&G(8,3)&G(8,4)&G(8,5)&G(8,6)&G(8,7)&G(8,8)&\\
$\mbox{\reflectbox{$\ddots$}}$&\vdots&\vdots&\vdots&\vdots&\vdots&\vdots&\vdots&\vdots&\vdots&$\ddots$
\end{tabular}
\end{table}
\begin{table}[h]
\centering
\begin{tabular}{cccccccccc}
&&&&G(1,0)&G(1,1)&&&&\\
&&&G(3,0)&G(3,1)&G(3,2)&G(3,3)&&&\\
&&G(5,0)&G(5,1)&G(5,2)&G(5,3)&G(5,4)&G(5,5)&&\\
&G(7,0)&G(7,1)&G(7,2)&G(7,3)&G(7,4)&G(7,5)&G(7,6)&G(7,7)&\\
$\mbox{\reflectbox{$\ddots$}}$&\vdots&\vdots&\vdots&\vdots&\vdots&\vdots&\vdots&\vdots&$\ddots$
\end{tabular}
\end{table}

\begin{lemma}[Triangle Induction]
\label{lem:triangle}
If $\scal>0$ then, for each $n\in\N_0$ and $k\in\{0,\ldots,n\}$:
$$G(n+2,k+1)> G(n,k).
$$
\end{lemma}

\begin{proof}
One directly computes
\begin{align*}
G(n+2,k+1)-G(n,k)&=4(c^2 n- bC +ac+ b^2+ c^2)>4(-bC+ac+b^2).
\end{align*}
If $\scal>0$ then this is positive by Lemma~\ref{lem:bounds}(v).
\end{proof}

\begin{remark}
\label{rem:suffbase}
In view of Lemma~\ref{lem:triangle} and the above diagrams, \eqref{eq:neven} and \eqref{eq:nodd} will follow from Proposition~\ref{prop:base} below; recall that $C\ge\mu>0$ by Lemma~\ref{lem:bounds}(ii), (iii).
\end{remark}

\begin{proposition}[Base cases for Triangle Induction]\
\label{prop:base}
If $\scal>0$ and $a\ge b\ge c$ then
\begin{itemize}
\item[(i)]
$G(0,0)=C^2$, and $G(n,n)>C^2$ for all $n\ge 1$,
\item[(ii)]
$G(n,0)>C^2$ for all $n\ge 6$,
\item[(iii)]
$G(n,1)>C^2$ for all $n\ge 4$,
\item[(iv)]
$G(1,0)=\mu^2$ and $G(5,0)>\mu^2$.
\end{itemize}
\end{proposition}

\begin{proof}

(i)
By definition we have
\begin{align*}
G(n,n)&=(a(n-2n)-C)^2+(b-c)^2 n-2(b-c)(C+a)n -(b^2-c^2)n(n-1)\\
&=(a^2-b^2+c^2)n^2+2(aC+b^2-bc-bC+cC-ab+ca)n+C^2.
\end{align*}
The coefficient at $n^2$ is positive since $a\ge b$. Recall that $-bC+b^2+ca>0$ by Lemma~\ref{lem:bounds}(v),
so the coeffient at~$n$ is greater than $2(a+c)(C-b)$, which is positive by Lemma~\ref{lem:bounds}(i).
The statement now follows.

\vskip5pt
(ii)
By definition,
\begin{equation}
\label{eq:gn0}
\begin{aligned}
G(n,0)&=(an-C)^2+(b+c)^2 n -2(b+c)(C-a)n-(b^2-c^2)n(n-1)\\
&=(a^2-b^2+c^2)n^2+2(-aC+b^2+bc-bC-cC+ab+ca)n+C^2.
\end{aligned}
\end{equation}
Again, the coefficient at $n^2$ is positive. If we view $n\mapsto G(n,0)$ as a quadratic function on~$\R$, we see that
$$n=\nu\coloneqq \frac{2(aC+bC+cC-b^2-ab-bc-ca)}{a^2-b^2+c^2}
$$
is the second solution (apart from $n=0$) of $G(n,0)=C^2$. Thus, it suffices to show that $\nu<6$. Equivalently, we have to
show that
$$f\coloneqq aC+bC+cC-b^2-ab-bc-ca-3(a^2-b^2+c^2)<0.
$$
Using Lemma~\ref{lem:bounds}(v) and $a\ge b$, we have
\begin{equation}
\label{eq:Ckey}
aC< a^2+bc, \ \ \ bC<b^2+ca,  \ \ \ cC<\tfrac ca (a^2+bc)\le ca+c^2.
\end{equation}
Thus,
\begin{align*}
f&<a^2+bc+b^2+ca+ca+c^2-b^2-ab-bc-ca-3(a^2-b^2+c^2)\\
&=-2a^2+3b^2+ca-2c^2-ab.
\end{align*}
We now use Lemma~\ref{lem:bounds}(iv) and $a\ge b\ge c$ to see that
\begin{equation}
\label{eq:ckey}
2c>\tfrac{2ab}{a+b}\ge b
\end{equation}
and therefore $f<b^2+ca-bc-ab=(b-c)(b-a)\le0$.

\vskip5pt
(iii)
By definition,
\begin{align*}
G(n,1)&=(a(n-2)-C)^2+(b-c)^2n+2(b+c)^2(n-1)\\
&\ \ \ \ -2(b-c)(C+a)-2(b+c)(C-a)(n-1)-(b^2-c^2)(n-1)(n-2)\\
&=C^2-\bigl((2n-4)a+2nb+(2n-4)c\bigr)C\\
&\ \ \ \ +(n-2)^2 a^2 -(n^2-6n+4)b^2+n^2 c^2
 +(2n-4)ab +(2n-4)bc  +2n\,ca.
\end{align*}
Using~\eqref{eq:Ckey} again, and keeping in mind that $n\ge4$ (in particular, $2n-4>0$), we obtain:
\begin{align*}
G(n,1)&>C^2-\bigl((2n-4)(a^2+bc)+2n(b^2+ca)+(2n-4)(ca+c^2)\bigr)\\
&\ \ \ \ +(n-2)^2 a^2 -(n^2-6n+4)b^2+n^2 c^2 +(2n-4)ab +(2n-4)bc +2n\,ca\\
&=C^2+(n^2-6n+8)a^2-(n^2-4n+4)b^2+(n^2-2n+4)c^2+(2n-4)ab-(2n-4)ca.
\end{align*}
Note that $n^2-6n+8\ge0$ since $n\ge4$. Therefore, using $a\ge b\ge c$ and~\eqref{eq:ckey} we conclude:
\begin{align*}
G(n,1)&>C^2+(n^2-6n+8)b^2-(n^2-4n+4)b^2\\
&\ \ \ \ +(n^2-2n+4-2(2n-4))c^2+(2n-4)bc+(2n-4)ab-(2n-4)ca\\
&=C^2-(2n-4)b^2+(n^2-6n+12)c^2+(2n-4)(bc+ab-ca)\\
&=C^2+(2n-4)(a-b)(b-c)+((n-3)^2+3)c^2>C^2.
\end{align*}

(iv)
Substituting $n=1$ in~\eqref{eq:gn0}, the statement $G(1,0)=(a+b+c-C)^2=\mu^2$ is immediate. For the second statement, we
let $n=5$ in~\eqref{eq:gn0} and use
\eqref{eq:Ckey} and~\eqref{eq:ckey} once more to obtain:
\begin{align*}
G(5,0)-\mu^2&=25(a^2-b^2+c^2)+10(-aC-bC-cC+b^2+ab+bc+ca)+C^2-\mu^2\\
&=-8aC-8bC-8cC+24a^2-16b^2+24c^2+8ab+8bc+8ac\\
&>8(-a^2-bc-b^2-ca-ca-c^2+3a^2-2b^2+3c^2+ab+bc+ca)\\
&=8(2a^2-3b^2+2c^2+ab-ca)> 8(-b^2+bc+ab-ca)=8(a-b)(b-c)\ge0.
\end{align*}
\end{proof}

\begin{remark}

(i)
The proof of Theorem~\ref{thm:firstexplicit} -- and, thus, Theorem~\ref{thm:first} -- is now complete; recall Remarks \ref{rem:suffbase} and \ref{rem:gershgoals}.

\vskip5pt
(ii)
Concerning Proposition~\ref{prop:base}, we add the observation that its statements are already optimal in the sense
that
\begin{itemize}
\item[$\bullet$] $G(2,0)$ and $G(4,0)$ are~\emph{not} always greater than or equal to~$C^2$\\
(and, less interestingly, neither is $G(5,0)$), and
\item[$\bullet$] $G(3,0)$ is \emph{not} always greater than or equal to~$\mu^2$.
\end{itemize}
In fact, consider the point $(a,b,c)=(1,1,\frac12)$. This corresponds to a metric $\gabc$ with $\scal=0$ and, thus,
does not satisfy our assumption $\scal>0$. However, in any of its neighborhoods there exist triples corresponding to metrics which do satisfy $\scal>0$; for example, $(1,1,\frac12+\eps)$ for small $\eps>0$.
Moreover, the formulas for $C$, $\mu$, and $G(n,k)$ are continuous in $(a,b,c)$. Thus, the above negative assertions follow from noting that
in $(1,1,\frac12)$ one has
$$C^2=\tfrac94,\ \ \mu^2=1,\ \ G(2,0)=\tfrac14,\ \ G(3,0)=0,\ \ G(4,0)=\tfrac14,\ \ G(5,0)=1.
$$
This behavior of $G(2,0)$, $G(3,0)$, $G(4,0)$ just says that it would not have been possible to prove the results for $n=2,3,4$ in
Remark~\ref{rem:smalln} by using the Gershgorin Circle Theorem.
Also, it would not be possible to show, using the Gershgorin Circle Theorem, that $\scal>0$ implies that the eigenvalues of $D_5^2$ are (even) greater than $C^2$ -- we do not know whether this is actually the case.

\vskip5pt
(iii) Finally, we remark that the eigenvalues of $D_3^2$ are smaller than $C^2$ for some metrics of positive scalar curvature. In fact, for small
$\eps>0$, the point $(a,b,c)=(1,\eps,\eps)$ corresponds to a metric with $\scal>0$. The eigenvalue of $D_3'$ listed under $\fbox6$ in Remark~\ref{rem:smalln}(v) has value $1-4\eps$ at that point, and $C=\frac12(2+\eps^2)$. The corresponding eigenvalue of $D_3=D_3'-C\Id$ is $1-4\eps-\frac12(2+\eps^2)=-4\eps-\frac12\eps^2$, so its square is smaller than $C^2$ if $\eps>0$ is sufficiently small.
\end{remark}

\Section{Spectral determination of isometry class}
\label{sec:invar}

\noindent
The aim of this section is to prove Theorem~\ref{thm:isometry}.

\begin{remark}
\label{rem:heatinvars}
Consider the asymptotic expansion
$$\Tr(e^{-tD^2})\sim_{t\searrow0}(4\pi t)^{-m/2}(a_0+a_1t+a_2t^2+\ldots)
$$
of the trace of the heat kernel associated with the square of the Dirac operator $D$
over a closed $m$-dimensional Riemannian spin manifold $(M,g)$.
The coefficients $a_j$ are determined by the eigenvalue spectrum of $D^2$ and, hence, by the eigenvalue spectrum of~$D$.

As proved in~\cite{DF},
\begin{equation}
\label{eq:heatinvars}
\begin{aligned}
a_0&=\dimm(\Sigma_m)\vol(M,g),\\
a_1&=-\frac{\dimm(\Sigma_m)}{12}\int_M\scal\,\dvol_g,\\
a_2&=\frac{\dimm(\Sigma_m)}{1440}\int_M(5\scal^2-8\|\Ric\|^2-7\|R\|^2)\dvol_g,
\end{aligned}
\end{equation}
where $\scal$, $\Ric$, $R$ are the  scalar curvature, the Ricci tensor, and the Riemannian curvature tensor of~$(M,g)$, respectively,
and $\Sigma_m$
is the space of complex $m$-spinors. (Alternatively, one can derive these equations as a special case of Gilkey's more general formulas for Laplace type operators $P$ on complex vector bundles in~\cite{Gil}, \S 4.8. Here, $P=D^2=\nabla^*\nabla+\frac14\scal$ by the Lichnerowicz formula, where $\nabla$ is the spinor connection on the spinor bundle $\Sigma M$.)

We apply this to our setting where $(M,g)$ is one of our manifolds $(S^3,\gabc)$ (with Dirac operator $D$ as in the
previous sections) or $(\SO(3),\gabc)$ (with Dirac operator either $D^{\alpha_0}$ or $D^{\alpha_1}$.
Since $\gabc$ is homogeneous, all integrands in~\eqref{eq:heatinvars} are constant. In particular, in each of the above cases we have:
\begin{itemize}
\item[$\bullet$] The volume is determined by $a_0$,
\item[$\bullet$] the scalar curvature is determined by $a_0$ and $a_1$,
\item[$\bullet$] the value of $8\|\Ric\|^2+7\|R\|^2$ is determined by $a_0$, $a_1$, and $a_2$.
\end{itemize}
\end{remark}

\begin{remark}
By the previous remark, the spectrum of any of our Dirac operators over $(S^3,\gabc)$ or $(\SO(3),\gabc)$ determines, in particular, whether $\scal>0$ or $\scal\le0$.
Our strategy for proving Theorem~\ref{thm:isometry} now is to show the following for each of the Dirac operators under consideration:
\begin{itemize}
\item[$\bullet$]
If $\scal>0$ then the volume and the scalar curvature, together with the eigenvalue of smallest absolute value, determine $a,b,c$ up to permutation.
\item[$\bullet$]
If $\scal\le0$ then the volume and the scalar curvature, together with $8\|\Ric\|^2+7\|\Riem\|^2$, determine $a,b,c$ up to permutation.
\end{itemize}
In view of Remark~\ref{rem:heatinvars} and Remark~\ref{rem:permut}, Theorem~\ref{thm:isometry} will then indeed be proven.
\end{remark}

\subsection{Case of positive scalar curvature}

\noindent
If $\scal>0$ then $\mu=a+b+c-C$ is determined by the spectrum of~$D$ over $(S^3,\gabc)$, and also by the spectrum
of~$D^{\alpha_1}$ over $(\SO(3),\gabc)$, since $\mu$ has smallest absolute value among all eigenvalues of~$D$ or $D^{\alpha_1}$
by Theorem~\ref{thm:first}. Note that
\begin{equation}
\label{eq:vol}
\vol(S^3,\gabc)=\frac{\vol(S^3,g_{111})}{abc}=\frac{2\pi^2}{abc}\text{\ \ and\ \ }\vol(\SO(3),\gabc)=\frac{\pi^2}{abc}\,.
\end{equation}
Recall that the scalar curvature
of $(S^3,\gabc)$ or $(\SO(3),\gabc)$ is constant and is given by the formula from Proposition~\ref{prop:scal}.

Showing that the volume together with $\scal$ and~$\mu$ determines $a,b,c$ up to permutation is equivalent to showing
that each of the symmetric elementary polynomials
$$ 
s_1\coloneqq a+b+c,\ \ \ s_2\coloneqq ab+bc+ca,\ \ \ s_3\coloneqq abc
$$
is individually determined by these data.
Clearly, \eqref{eq:vol} implies that $s_3$ is determined by the volume of $(S^3,\gabc)$, resp.~$(\SO(3),\gabc)$.
Moreover,
\begin{equation}
\begin{aligned}
\label{eq:musymm}
\mu&=a+b+c-\tfrac12\left(\tfrac{ab}c+\tfrac{bc}a+\tfrac{ca}b\right)
   =s_1-\tfrac12(a^2b^2+b^2c^2+c^2a^2)/s_3
	 =s_1-\tfrac12(s_2^2-2s_1s_3)/s_3\\
	 &=2s_1-\tfrac12 s_2^2/s_3
\end{aligned}
\end{equation}
and, by Proposition~\ref{prop:scal},
\begin{equation}
\begin{aligned}
\scal&=4(a^2+b^2+b^2)-2\left(\tfrac{a^2b^2}{c^2}+\tfrac{b^2c^2}{a^2}+\tfrac{c^2a^2}{b^2}\right)
  =4(s_1^2-2s_2)-2(a^4b^4+b^4c^4+c^4a^4)/s_3^2\\
	&=4(s_1^2-2s_2)-2\bigl((s_2^2-2s_1s_3)^2-2s_3^2(s_1^2-2s_2)\bigr)/s_3^2
	=-16s_2-2s_2^4/s_3^2+8s_1s_2^2/s_3\\
	&=-16s_2+4\mu\,s_2^2/s_3.
\end{aligned}
\end{equation}
We read the resulting formula
$$
4\mu\,s_2^2/s_3-16s_2-\scal=0
$$
as a quadratic equation for $s_2$. Because of $\mu>0$, $s_3>0$, $\scal>0$ there exists precisely one \emph{positive} solution $s_2$ (and precisely one negative solution, which is irrelevant since $s_2>0$). Thus, $s_2$ is determined by $s_3$, $\mu$, and $\scal$.
Now \eqref{eq:musymm} implies that this is the case for $s_1$, too.

For $D^{\alpha_0}$ over $(\SO(3),\gabc)$ we argue analogously. Instead of~$\mu$, we use
$C=\frac12\left(\frac{ab}c+\frac{bc}a+\frac{ca}b\right)$, which is the smallest absolute value
of eigenvalues of $D^{\alpha_0}$ over $\SO(3)$ by Theorem~\ref{thm:first}.
Here, it will be more convenient
to use the elementary symmetric polynomials in $a^2,b^2,c^2$:
$$\sigma_1\coloneqq a^2+b^2+c^2,\ \ \ \sigma_2\coloneqq a^2b^2+b^2c^2+c^2a^2,\ \ \ \sigma_3\coloneqq a^2b^2c^2.
$$
These determine $a^2,b^2,c^2$ and, thus, also their positive roots $a,b,c$, up to permutation.
Clearly, $\sigma_3$ is determined by $\vol(\SO(3),\gabc)=\pi^2/\sqrt{\sigma_3}$.
Next, $\sigma_2$ is determined by $C$ and $\sigma_3$ since
\begin{equation}
\label{eq:Csym}
C=\tfrac12\sigma_2/\sqrt{\sigma_3}.
\end{equation}
Finally, 
\begin{equation}
\label{eq:scalsym2}
\scal=8(\sigma_1-C^2)
\end{equation}
by Proposition~\ref{prop:scal}; in particular, $\sigma_1$ is determined by $C$ and $\scal$. This concludes the proof of Theorem~\ref{thm:isometry} in the case of positive scalar curvature.

\subsection{Case of nonpositive scalar curvature}

\noindent
Here we will be able to argue simultaneously for $D$, $D^{\alpha_0}$, and $D^{\alpha_1}$.
Instead of the smallest absolute value of eigenvalues, which is not at our disposal if $\scal\le0$, we now use
$$\tilde a_2\coloneqq 8\|\Ric\|^2+7\|\Riem\|^2
$$
which is spectrally determined by Remark~\ref{rem:heatinvars}.
Since we are in dimension three, $\Riem$ is a linear combination of $\Ric\KN g$ and $g\KN g$ (where $\KN$ denotes the Kulkarni-Nomizu product), which implies that $\Riem(X,Y)Z=0$ if $\{X,Y,Z\}$ is an orthogonal basis consisting of Ricci eigenvectors. Since the latter is the case for our Milnor basis $\{X_1,X_2,X_3\}$ from Definition~\ref{def:gabc}, we have, writing $K_{ij}\coloneqq K(\spann\{X_i,X_j\})$ for the sectional curvatures:
$$\|\Riem\|^2=4(K_{12}^2+K_{23}^2+K_{31}^2)
$$
(the factor 4 arising from $R_{ijji}^2+R_{ijij}^2+R_{jiji}^2+R_{jiij}^2=4K_{ij}^2$)
and, of course,
$$\|\Ric\|^2=(K_{12}+K_{23})^2+(K_{23}+K_{31})^2+(K_{31}+K_{12})^2.
$$
By the formulas on p.~319/320 in~\cite{Mi} and recalling that, in our notation, Milnor's $\lambda_1,\lambda_2,\lambda_3$ are
$\frac{2bc}a$, $\frac{2ca}b$, and $\frac{2ab}c$, respectively, one has
$$K_{12}=2(a^2+b^2-c^2)+\frac{b^2c^2}{a^2}+\frac{c^2a^2}{b^2}-3\frac{a^2b^2}{c^2}
$$
and the analogous formulas for $K_{23}$ and $K_{31}$. By direct computation, using the elementary symmetric polynomials
in $a^2,b^2,c^2$ again, this leads to
\begin{align*}
\|\Ric\|^2&=64\sigma_1^2-64\sigma_1\sigma_2^2/\sigma_3+12\sigma_2^4/\sigma_3^2+64\sigma_2\text{ \ \ \ \ and}\\
\|\Riem\|^2&=192\sigma_1^2-224\sigma_1\sigma_2^2/\sigma_3+44\sigma_2^4/\sigma_3^2+256\sigma_2,
\end{align*}
hence 
\begin{equation}
\label{eq:ricandr}
\tilde a_2=1856\sigma_1^2-2080\sigma_1\sigma_2^2/\sigma_3+404\sigma_2^4/\sigma_3^2+2304\sigma_2.
\end{equation}
Moreover, recall from~\eqref{eq:scalsym2} and \eqref{eq:Csym} that
\begin{equation}
\label{eq:scal}
\scal=8\sigma_1-2\sigma_2^2/\sigma_3.
\end{equation}
As before, $\sigma_3$ is determined by the volume of the underlying Riemannian manifold. It now suffices to show that
$\sigma_2$ is determined by $\scal$ and $\tilde a_2$ because, by~\eqref{eq:scal}, $\sigma_1$ will then be determined by these and the volume.

Using \eqref{eq:ricandr}, \eqref{eq:scal}, and the corresponding equation for $\scal^2$, we obtain that
\begin{align*}
Y\coloneqq (\tilde a_2-101\scal^2)/576
&= -8\sigma_1^2+2\sigma_1\sigma_2^2/\sigma_3+4\sigma_2\\
&= -\scal\cdot\sigma_1+4\sigma_2
\end{align*}
is determined by $\scal$ and $\tilde a_2$.
In case $\scal=0$, it follows immediately that $\sigma_2$, too, is determined by these data, and we are done.
In case $\scal\ne0$, we have $\sigma_1=(4\sigma_2-Y)/\scal$. Substituting this for~$\sigma_1$ in~\eqref{eq:scal}, we obtain
\begin{equation}
\label{eq:last}
\scal=8\dfrac{4\sigma_2-Y}{\scal}-2\dfrac{\sigma_2^2}{\sigma_3}
\end{equation}
We read this as a quadratic equation for~$\sigma_2$ with coefficients determined by the volume together with $\scal$ and $\tilde a_2$. Under the assumption $\scal<0$, the right hand side of~\eqref{eq:last} has strictly negative derivative in each $\sigma_2>0$, so there cannot be more than one positive solution~$\sigma_2$. (Note that this argument would not have worked if $\scal>0$.)
This concludes the proof of Theorem~\ref{thm:isometry} in the case of nonpositive scalar curvature.

\section*{Acknowledgments}

The second author thanks Emilio Lauret for raising her attention to the topic, and for several interesting discussions.

\end{document}